\documentclass{endm}
\usepackage{endmmacro}
\usepackage{graphicx}
\usepackage{amssymb}

\usepackage{tikz}
\usepackage{tikz-cd}
\usetikzlibrary{arrows,automata}
\usepackage[utf8]{inputenc}
\usepackage[T1]{fontenc}
\usepackage{amsmath,amssymb}
\newcommand\rsetminus{\mathbin{\mathpalette\rsetminusaux\relax}}
\newcommand\rsetminusaux[2]{\mspace{-4mu}
  \raisebox{\rsmraise{#1}\depth}{\rotatebox[origin=c]{-20}{$#1\smallsetminus$}}
 \mspace{-4mu}
}
\newcommand\rsmraise[1]{%
  \ifx#1\displaystyle .8\else
    \ifx#1\textstyle .8\else
      \ifx#1\scriptstyle .6\else
        .45%
      \fi
    \fi
  \fi}
\usepackage{mathptmx}       
\usepackage{helvet}         
\usepackage{courier}        
\usepackage{type1cm}        
\usepackage{xhfill}         
\usepackage{xpatch}
\usepackage{textcase}
\makeatletter
\xpatchcmd{\@sect}{\uppercase}{\MakeTextUppercase}{}{}
\xpatchcmd{\@sect}{\uppercase}{\MakeTextUppercase}{}{}
\makeatother
\usepackage{makeidx}         
\usepackage{graphicx}        
\usepackage{multicol}        
\usepackage{pgf,tikz}
\usetikzlibrary{arrows}
\usepackage{ulem}

\def\scal#1#2{\langle #1\bv#2 \rangle}

\def\ncp#1#2{#1\langle #2\rangle}
\def\ncs#1#2{#1\langle \!\langle #2\rangle \!\rangle}

\usepackage{textcomp}
\usepackage{enumerate}
\usepackage{tikz}
\usetikzlibrary{arrows,decorations.pathmorphing}
\usetikzlibrary{backgrounds,positioning,fit,petri}
\usetikzlibrary{matrix}
\def\calX{{\cal X}}

\def\calG{{\mathcal G}}
\def\calX{{\mathcal X}}

\def\QY{\ncp{\Q}{Y}}

\def\scal#1#2{\langle #1\bv#2 \rangle}
\def\bv{\mid}

\def\abs#1{\bv\!#1\!\bv}

\def\ncp#1#2{#1\langle #2\rangle}

\def\conc{{\tt conc}}

\def\Dom{\mathrm{Dom}}
\def\der{\mathbf{d}}
\newcommand\matop[2]{\genfrac{}{}{0pt}{}{#1}{#2}}

\newcounter{per1}
\begingroup
\def\2#1{\ifnum#1<10 0\fi\the#1}
\xdef\isodayandtime{
{\2\day-\2\month-\the\year\space\2{\count0}:%
\2{\count2}}}
\endgroup
\newcommand{\bi}{\begin{itemize}}
\newcommand{\ei}{\end{itemize}}
\newcommand{\bd}{\begin{description}}
\newcommand{\ed}{\end{description}}

\usepackage{amsmath}
\usepackage{amsfonts}
\usepackage{amssymb}

\newcommand{\calA}{{\mathcal A}}
\newcommand{\calZ}{{\mathcal Z}}

\newcommand{\calH}{{\mathcal H}}

\newcommand{\N}{{\mathbb N}}

\newcommand{\Z}{{\mathbb Z}}
\newcommand{\Q}{{\mathbb Q}}

\newcommand{\C}{{\mathbb C}}

\def\abs#1{|#1|}
\def\absv#1{\|#1\|}

 \def\shuffle{\mathop{_{^{\sqcup\!\sqcup}}}} 

\gdef\stuffle{\;%
  \setlength{\unitlength}{0.0125cm}%
  \begin{picture}(20,10)(220,580) 
  \thinlines 
  \put(220,592){\line( 0,-1){ 10}} 
  \put(220,582){\line( 1, 0){ 20}} 
  \put(240,582){\line( 0, 1){ 10}} 
  \put(230,592){\line( 0,-1){ 10}} 
  \put(225,587){\line( 1, 0){ 10}} 
  \end{picture}\; 
}

\newcommand{\Li}{\operatorname{Li}}

\def\L{\mathrm{L}}
\def\H{\mathrm{H}}

\newcommand{\serie}[2]{#1 \langle \! \langle #2 \rangle \! \rangle}

\def\CXX{\serie{\C}{X}}

\def\QY_0{\Q\left\langle{Y_0}\right\rangle}

\def\CXX{\serie{\C}{X}}
 
 \newcommand{\calB}{{\cal B}}

\def\path{\rightsquigarrow}

\def\bv{\mid}

\def\abs#1{\bv\!#1\!\bv}

\usepackage{color}
\usepackage{ulem}

\def\scal#1#2{\langle #1\bv#2 \rangle}
\def\ncp#1#2{#1\langle #2\rangle}
\def\ncs#1#2{#1\langle \!\langle #2\rangle \!\rangle}

\begin{document}
\begin{frontmatter}
\title{A local Theory of Domains and its (Noncommutative) Symbolic Counterpart}
\author{V.C. Bui}
\address{Hue University of Sciences, 77 - Nguyen Hue street - Hue city, Vietnam.}
\author{G.H.E. Duchamp}
\address{University Paris 13, Sorbonne Paris City, 93430 Villetaneuse, France,}
\author{V. Hoang Ngoc Minh}
\address{University of Lille, 1 Place Déliot, 59024 Lille, France,} 
\author{Q.H. Ngo}
\address{University of Hai Phong, 171, Phan Dang Luu, Kien An, Hai Phong, Viet Nam} 
\author{K. A. Penson}
\address{Sorbonne Universit\'e, Universit\'e Paris VI, 75252 Paris Cedex 05, France.} 
\date{{\small (Draft)}}
\newpage
\begin{abstract}
It is widely accepted nowadays that polyzetas are connected by polynomial relations.
One way to obtain relations among polyzetas is to consider their generating series
and the relations among the coefficients of these generating series. 
This leads to the re-indexation of the coefficients of the generating series of polylogarithms 
itself, recently mentioned in \cite{BHN,CM} and, in a way, this work is a continuation of \cite{GHM22}.
But, in order to understand the bridge between this extension of this ``polylogarithmic calculus''
and the world of harmonic sums, a local theory of domains has to be done,
preserving quasi-shuffle identities, Taylor expansions and Hadamard products.
In this contribution, we present a sketched and abridged version of this theory.    

As an example of generating series, one can consider the eulerian gamma function, 
\begin{eqnarray*}
\Gamma(1+z)=\exp\biggl(-\gamma z+\sum_{n\ge2}\zeta(n)\dfrac{(-z)^n}{n}\biggr)
\end{eqnarray*}
and this may suggest to regularize the divergent zeta value $\zeta(1)$, for the quasi-shuffle structure,
as to be Euler's $\gamma$ constant.
In the same vein, in \cite{BHN}, we introduce a family of eulerian functions,
\begin{eqnarray*}
\Gamma_{y_k}(1+z)=\exp\biggl(\sum_{n\ge1}\zeta(kn)\dfrac{(-z^k)^n}{n}\biggr),
&\mbox{for}&k\ge2,y_k\in Y=\{y_n\}_{n\ge1}.
\end{eqnarray*}
This being done, in this work, via their analytical aspects, we establish, on one side, their existence 
and the fact that their inverses are entire. On the other side, using the same symmetrization
technique, we give their distributions of zeroes\footnote{This research is funded by the Vietnam National
Foundation for Science and Technology Development (NAFOSTED) under grant number  101.04-2017.320}.
\end{abstract}
\end{frontmatter}
\tableofcontents

\section{Introduction}\label{intro}
This work is partly the continuation of \cite{GHM22,BHN,CM} where it has been established that the
polylogarithms, indexed by the $r$-tuples $(s_1,\ldots,s_r)\in\C^r$, are well defined locally by 
\begin{eqnarray}\label{polylogarithm}
\Li_{s_1,\ldots,s_r}(z):=\sum_{n_1>\ldots>n_r>0}\frac{z^{n_1}}{n_1^{s_1}\ldots n_r^{s_r}},\mbox{for}&|z|<1,
\end{eqnarray}
could be extended, in case $(s_1,\ldots,s_r)\in\N_+^r$, to some series, over the alphabet $X=\{x_0,x_1\}$
generating the monoid $X^*$ with the neutral element $1_{X^*}$ \cite{berstel}. More precisely,
\begin{enumerate}
\item we start to consider \cite{HNM1}
\begin{eqnarray}
\forall w=x_0^{s_1-1}x_1\ldots x_0^{s_r-1}x_1\in X^*x_1,&&\Li_w=\Li_{s_1,\ldots,s_r},
\end{eqnarray}
\item then extend\footnote{This paper uses extensively shuffle and stuffle products (noted $\shuffle$
and $\stuffle$ respectively). For readers unfamiliar with these subjects their definitions are recalled at
the end of this text, see paragraph \ref{stuffledef}.} $\Li_{\bullet}$ as the $\shuffle$-morphism
$(\ncp{\C}{X},\shuffle,1_{X^*})\longrightarrow(\C\{\Li\}_{w\in X^*},\times,1)$ by adding
$\Li_{x_0}(z)=\log(z)$. This morphism is injective and satisfies \cite{SLC43}
\begin{eqnarray}\label{shuffle}
\forall S,T\in\ncp{\C}{X},&&\Li_{S\shuffle T}=\Li_{S}\Li_{T}.
\end{eqnarray}
\item For the sake of symbolic calculations, it is important that, on the one hand, these series should belong
to some ``computable spaces'' and, on the other hand, that the new domain (a) be closed by shuffle products
and (b) that the $\Li_\bullet$ correspondence should preserve the shuffle identity \eqref{shuffle}.
\end{enumerate}
To this end a theory of global domains was presented in \cite{GHM22,BHN,CM}.
Here we focus on what happens in the neighbourhood of zero, therefore, the aim
of this work is manyfold. Let us highligh the many facets of this matter.
\begin{enumerate}
\item Propagate the extension to local Taylor expansions\footnote{Around zero.} as in \eqref{polylogarithm}
and the coefficients of their quotients by $1-z$, namely the harmonic sums, denoted $\H_\bullet$ and defined,
for any $w\in X^*x_1$, as follows\footnote{Here, the ${\tt conc}$-morphism
$\pi_X:(\ncp{\C}{Y},{\tt conc},1_{Y^*})\longrightarrow(\ncp{\C}{X},{\tt conc},1_{X^*})$
is defined by $\pi_X(y_n)=x_0^{n-1}x_1$ and $\pi_Y$ its inverse on $\mathrm{Im}(\pi_X)$.
See \cite{GHM22,BHN,CM} for more details and a full definition of $\pi_Y$.} \cite{words03}
\begin{eqnarray}
\frac{\Li_w(z)}{1-z}=\sum_{N\ge0}\H_{\pi_X(w)}(N)z^N,
\end{eqnarray}
by a suitable theory of local domains which assures to carry over the computation of these Taylor coefficients
and preserves the stuffle indentity, again true for polynomials over the alphabet $Y=\{y_n\}_{n\ge1}$, \textit{i.e.}
\begin{eqnarray}
\forall S,T\in\ncp{\C}{Y},&&\H_{S\stuffle T}=\H_{S}\H_{T},
\end{eqnarray}
meaning that $\H_{\bullet}:(\ncp{\C}{Y},\stuffle,1_{Y^*})\longrightarrow(\C\{\H_w\}_{w\in Y^*},\times,1)$,
mapping any word $w=y_{s_1}\ldots y_{s_r}\in Y^*$ to
\begin{eqnarray}\label{HarmonicSums}
\H_w=\H_{s_1,\ldots,s_r}=\sum_{N\ge n_1>\ldots>n_r>0}\frac1{n_1^{s_1}\ldots n_r^{s_r}},
\end{eqnarray}
is a injective $\stuffle$-morphism \cite{words03}.

\item Extend these correspondences (\textit{i.e.} $\Li_\bullet,\H_\bullet$) to some series (over $X$ and $Y$, respectively)
in order to preserve the identity\footnote{Here $\odot$ stands for the Hadamard product \cite{Had_prod}.} \cite{words03}
\begin{eqnarray}\label{idP}
\frac{\Li_{\pi_X(S)}(z)}{1-z}\odot\frac{\Li_{\pi_X(T)}(z)}{1-z}=\frac{\Li_{\pi_X(S\stuffle T)}(z)}{1-z}.
\end{eqnarray}
true for polynomials $S,T\in\ncp{\C}{Y}$.

\item Taking the definition of polyetas as in \eqref{polylogarithm} at $z=1$ or in \eqref{HarmonicSums} at $+\infty$,
one sees that, for any $s_1>1$, Abel's theorem, one has
\begin{eqnarray}
\zeta(s_1,\ldots,s_r)
&=&\lim_{z\rightarrow1}\Li_{s_1,\ldots,s_r}(z)\\
&=&\lim_{N\rightarrow+\infty}\H_{s_1,\ldots,s_r}(N)\\
&=&\sum_{n_1>\ldots>n_r>0}\frac1{n_1^{s_1}\ldots n_r^{s_r}}.
\end{eqnarray}    
However, this theorem does not hold in the divergent cases. and we will recall some regularization process
based on the computation of a $\stuffle$-character with polynomial values and specialize it to obtain a character \cite{Daresbury,Daresbury1}
\begin{eqnarray}
\gamma_\bullet:(\ncp{\Q}{Y},\stuffle,1_{Y^*})\longrightarrow(\calZ[\gamma],\times,1),
\end{eqnarray}    
where ${\cal Z}:=\mathrm{span}_{\Q}\{\zeta(s_1,\ldots,s_r)\}_{r\ge1,s_1\ge2,s_2,\ldots,s_r\ge1}$.
\item To this end, we use the explicit parametrization of the $\tt conc$-characters obtained in \cite{GHM22,BHN,CM} and the fact that,
under stuffle products, they form a group. We show the linear independence of the Kleene stars $(z^ky_k)^*$ and show that
$\gamma_\bullet$ provides a group morphism between the group of $\tt conc$-characters (endowed with $\stuffle$) and that of Taylor
series $g$ (with radius $R=1$) such that $g(0)=1$. This morphism maps each star $y_k^*$ precisely to 
\begin{eqnarray}
\frac{1}{\Gamma_{y_k}(1+z)}=\exp\biggl(-\sum_{n\ge1}\zeta(kn)\frac{(-z^k)^n}n\biggr),&\mbox{for}&k\ge2.
\end{eqnarray}    
and $y_1^*$ to the classical inverse Gamma:
\begin{eqnarray}
\Gamma_{y_1}^{-1}(1+z)=\Gamma^{-1}(1+z).
\end{eqnarray}    
We will prove that all these ``new'' functions are entire and linearly independant. 
\end{enumerate} 

To summarize, the present work concerns the whole project of extending $\H_{\bullet}$
over a stuffle subalgebra of rational power series on the alphabet $Y$, in particular
the stars of letters and some explicit combinatorial consequences of this extension.

\section{Domains and extensions}
All starts with the (multiindexed) polylogarithm defined, for $|z|<1$, by \eqref{polylogarithm}.
It is (multi-)indexed by a list $(s_1,\ldots,s_r)\in\N_{\ge1}^r$ which can be reindexed by a word 
$x_0^{s_1-1}x_1\ldots x_0^{s_r-1}x_1\in X^*x_1$. From this, introducing two differential forms
\begin{eqnarray}
\omega_0(z)=z^{-1}dz&\mbox{and}&\omega_1(z)=(1-z)^{-1}dz,
\end{eqnarray}
we get an integral representation of the functions \eqref{polylogarithm} as follows\footnote{Given
a word $w\in X^*$,  we note $|w|_{x_1}$ the number of occurrences of $x_1$ within $w$.} \cite{HNM1}
\begin{eqnarray}
\Li_w(z)=\left\{
\begin{array}{lclcl}
1_{\calH(\Omega)}&\mbox{if}&w=1_{X^*}\cr
\displaystyle\int_{0}^z\omega_1(s)\Li_u(s)&\mbox{if}& w=x_1u \cr
\displaystyle\int_{1}^z\omega_0(s)\Li_u(s)&\mbox{if}&w=x_0u&\mbox{and}&|u|_{x_1}=0,w\in x_0^*\cr
\displaystyle\int_{0}^z\omega_0(s)\Li_u(s)&\mbox{if}&w=x_0u&\mbox{and}&|u|_{x_1}>0,w\notin x_0^*,
\end{array}\right.
\end{eqnarray}
where $\Omega$ is the simply connected domain $\C\setminus(]-\infty,0]\cup[1,+\infty[)$,
over which we consider the algebra of analytic functions, $\calH(\Omega)$,
with the neutral element $1_{\calH(\Omega)}$. This provides not only the analytic
continuation of \eqref{polylogarithm} to $\Omega$ but also extends the indexation
to the whole alphabet $X$, allowing to study the complete generating series
\begin{eqnarray}
\L(z)=\sum_{w\in X^*}\Li_w(z)w
\end{eqnarray}
and show that it is the solution of the following first order noncommutative differential equation
\begin{eqnarray}\label{DrinfeldSys}
\left\{
\begin{array}{lcl}
\der(S)=(\omega_0(z)x_0+\omega_1(z)x_1)S,&&(NCDE)\cr
\lim\limits_{z\in\Omega,z\to 0}S(z)e^{-x_0\log(z)}=1_{\ncs{\calH(\Omega)}{X}},
&&\mbox{asymptotic initinial condition,}
\end{array}\right.
\end{eqnarray}
where, for any $S\in\ncs{\calH(\Omega)}{X}$, for term by term derivation, one gets \cite{VD1}
\begin{eqnarray}
\der(S)=\sum_{w\in X^*}\dfrac{d}{dz}(\scal{S}{w})w.
\end{eqnarray}
This differential system allows to show that $\L$ is a $\shuffle$-character \cite{SLC43}, \textit{i.e.}
\begin{eqnarray}
\forall u,v\in X^*,\quad\scal{\L}{u\shuffle v}=\scal{\L}{u}\scal{\L}{v}
&\mbox{and}&\scal{\L}{1_{X^*}}=1_{\calH(\Omega)}.
\end{eqnarray}

Note that, in what precedes, we used the pairing  $\scal{\bullet}{\bullet}$ between series and polynomials,
classically defined by, for $T\in\ncs{\C}{X}$ and $P\in\ncp{\C}{X}$\footnote{Here $R$ is any commutative ring
(like $\calH(\Omega),\C,\calZ[\gamma]$, ...).}  
\begin{eqnarray}
\scal{T}{P}=\sum_{w\in X^*}\scal{T}{w}\scal{P}{w}, 
\end{eqnarray}
where, when $w$ is a word, $\scal{S}{w}$ stands for the coefficient of $w$ in $S$. With this at hand,
we extend  at once the indexation of $\Li$ from $X^*$ to $\ncs{\C}{X}$ by 
\begin{eqnarray}\label{linext}
\Li_{P}:=\sum_{w\in X^*}\scal{P}{w}\Li_w=\sum_{n\ge0}\biggl(\sum_{|w|=n}\scal{P}{w}\Li_w\biggr).
\end{eqnarray}

In \cite{GHM22,BHN,CM}, it has been established that the polylogarithm, well defined locally by \eqref{polylogarithm},
could be extended to some series (with conditions) by the last part of formula \eqref{linext} where the polynomial
$P$ is replaced by some series.
 
As was said previously, we focus here on what happens in the neighbourhood of zero. Therefore, the aim of
this paragraph concerns the two first points of Section \ref{intro}. which we summarize here
\begin{enumerate}
\item Propagate the extension to local Taylor expansions\footnote{Around zero.}
of polylogarithms and the coefficients of their quotients by $1-z$, namely the
harmonic sums, by a suitable theory of local domains. 
\item Extend these correspondences (\textit{i.e.} $\Li_\bullet,\H_\bullet$) to some series in order to
preserve the identity \eqref{idP}.   
\end{enumerate} 

\subsection{Polylogarithms: from global to local domains}
The map $\Li_{\bullet}$ in general has been extended to a subdomain of $\ncs{\C}{X}$,
called $\Dom(\Li_{\bullet})$ (see \cite{GHM22,BHN,CM}). It is the set of series
\begin{eqnarray}
S=\sum\limits_{n\ge0}S_n,&\mbox{where}&S_n:=\sum\limits_{|w|=n}\scal{S}{w}
\end{eqnarray}
such that $\sum\limits_{n\ge0}\Li_{S_n}$ is unconditionally
convergent for the standard topology on $\calH(\Omega)$ \cite{Sch}.

\begin{example}[\cite{HNM1}]
For example, the classical polylogarithms: dilogarithm $\Li_2$, trilogarithm $\Li_3$,
etc... are defined and obtained through this coding by 
\begin{eqnarray*}
\Li_k(z)=\sum_{n\ge1}\frac{z^n}{n^k}=\Li_{x_0^{k-1}x_1}(z)=\scal{\L(z)}{x_0^{k-1}x_1}
\end{eqnarray*}
but for $t\ge0$ (real), the series $(tx_0)^*x_1$ belongs to $\Dom(\Li_\bullet)$ iff $0\le t<1$.
\end{example}

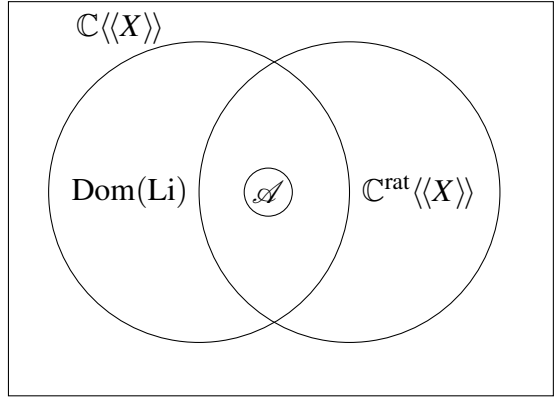
\begin{center}
\begin{tikzpicture}
\def\radius{2cm}
\def\mycolorbox#1{\textcolor{#1}{\rule{2ex}{2ex}}}
\colorlet{colori}{blue!60}
\colorlet{colorii}{red!60}
%
\coordinate (ceni);
\coordinate[xshift=.9\radius] (ceniii);
\coordinate[xshift=\radius] (cenii);

\draw (ceni) circle (\radius);
\draw (cenii) circle (\radius);
\draw (ceniii) circle (0.3\radius);

\draw  ([xshift=-15pt,yshift=15pt]current bounding box.north west) 
  rectangle ([xshift=20pt,yshift=-20pt]current bounding box.south east);

\node[xshift=-.9\radius] at (ceni) {$\Dom(\Li)$};
\node[xshift=.9\radius] at (cenii) {$\ncs{\C^{\mathrm{rat}}}{X}$};
\node[xshift=.9\radius] at (ceni) {$\calA$};
\node[xshift=-30pt,yshift=\radius+5pt] at (ceni) {$\ncs{\C}{X}$};
\end{tikzpicture}
\end{center}
\begin{center}
Above $\calA=\ncp{\C}{X}\shuffle\ncs{\C^\mathrm{rat}}{X}$
and $\ncs{\C^{\mathrm{rat}}}{X}$\\
is the set of rational series \cite{GHM22,BHN,CM}.
\end{center}
\medskip

This definition has many merits\footnote{As the fact that, due to special properties of $\calH(\Omega)$
(it is a nuclear space \cite{Sch}, see details in \cite{GHM22,BHN,CM}), one can show that
$\Dom(\Li)$ is closed by shuffle products.} and can easily be adapted to arbitrary (open and connected)
domains. But this definition, based on a global condition of a fixed domain $\Omega$, does not provide
a sufficiently clear  interpretation of the stable symbolic computations around a point, in particular
at $z=0$. One needs to consider a sort of ``symbolic local germ'' worked out explicitely. Indeed,
as the harmonic sums (or MZV) are the coefficients of the Taylor expansion at zero of the convergent
polylogarithms divided by $1-z$, we only need to know locally these functions. In order to gain more
indexing series and to describe the local situation at zero, we reshape and define a new domain of $\Li$
around zero to $\Dom^{\mathrm{loc}}(\Li_{\bullet})$. The first step will be provided by the following theorem.
\begin{theorem}\label{th1}
Let $S\in \ncs{\C}{X}x_1\oplus \C1_{X^*}$ such that
\begin{eqnarray*}
S=\sum_{n\ge0}[S]_n&\mbox{where}&[S]_n=\sum_{w\in X^*,\abs{w}=n}\scal{S}{w}w,
\end{eqnarray*}
($[S]_n$ are the homogeneous components of $S$), we suppose that $0<R\le1$
and that $\sum\limits_{n\ge0}\Li_{[S]_n}$ is unconditionally convergent
(for the standard topology) within the open disk $\abs{z}<R$.
Remarking that $\dfrac{1}{1-z}\sum\limits_{n\ge0}\Li_{[S]_n}(z)$
is unconditionally convergent in the same domain, we set 
\begin{eqnarray*}
\frac{1}{1-z}\sum_{n\ge0}\Li_{[S]_n}(z)=\sum_{N\ge0}a_Nz^N.
\end{eqnarray*}
Then, for all $N\ge0$,
\begin{eqnarray*}
\sum\limits_{n\ge0}\H_{\pi_Y([S]_n)}(N)=a_N.
\end{eqnarray*}
\end{theorem}

\begin{proof} 
Let us recall that, for any $w \in X^*$, the function $(1-z)^{-1}\Li_w(z)$ is analytic in the open disk $|z|<R$. Moreover, one has 
\begin{eqnarray*}
\frac{1}{1-z}\Li_w(z)=\sum_{N\ge0}\H_{\pi_Y(w)}(N)z^N.
\end{eqnarray*}
Since $[S]_n=\sum\limits_{w\in X^*,\abs{w}=n}\scal{S}{w}w$ and $(1-z)^{-1}\sum\limits_{n\ge0}\Li_{[S]_n}$ 
absolutely converges (for the standard topology\footnote{For this topology, unconditional
and absolute convergence coincide \cite{Sch}}) within the open disk $|z|<R$, one obtains 
\begin{eqnarray*}
\frac{1}{1-z}\sum_{n\ge0}\Li_{[S]_n}(z)
&=&\frac{1}{1-z}\sum_{n\ge0}\sum_{w\in X^*,\abs{w}=n}\scal{S}{w}w\Li_w(z)\cr
&=&\sum_{n\ge0}\sum_{w\in X^*,\abs{w}=n}\scal{S}{w}w\frac{\Li_w(z)}{1-z}\cr
&=&\sum_{n\ge0}\sum_{w\in X^*,\abs{w}=n}\scal{S}{w}w\sum_{N\ge0}\H_{\pi_Y(w)}(N)z^N\cr
&=&\sum_{N\ge0}\sum_{n\ge0}\sum_{w\in X^*,\abs{w}=n}\scal{S}{w}w\H_{\pi_Y(w)}(N)z^N\cr
&=&\sum_{N\ge0}\H_{\pi_Y([S]_n)}(N)z^N.
\end{eqnarray*}
This implies that, for any $N\ge0$,
\begin{eqnarray*}
a_N=\sum_{n\ge0}\H_{\pi_Y([S]_n)}(N).
\end{eqnarray*}
\end{proof}

We will need the following combinatorial

\begin{lemma}\label{surj_comb}
For a letter ``$a$'', one has 
\begin{eqnarray}\label{shuffle_pow}
\abs{(a^+)^{\shuffle m}}{a^n}=m!S_2(n,m)
\end{eqnarray}
($S_2(n,m)$ being the Stirling numbers of the second kind).
The exponential generating series of R.H.S. in equation \eqref{shuffle_pow} (w.r.t. $n$) is given by
\begin{eqnarray}
\label{12fold1}
\sum_{n\ge0} m! S_2(n,m)\frac{x^n}{n!}=(e^x-1)^m.
\end{eqnarray}
\end{lemma}

\begin{proof}$(a^+)^{\shuffle m}$ is the specialization of 
\begin{eqnarray*}
L_m=a_1^+\shuffle a_2^+\shuffle\ldots\shuffle a_m^+
\end{eqnarray*}
to $a_j\to a$ (for all $j=1,2\ldots m$). The words of $L_m$ are in bijection with the surjections 
$[1\ldots n]\to [1\ldots m]$, therefore the coefficient $\left\langle  (a^+)^{\shuffle m} |  a^n \right\rangle $ is exactly 
the number of such surjections namely $m!S_2(n,m)$. A classical formula\footnote{See \cite{St1},
the twelvefold way, formula (1.94b)(pp. 74) for instance.} says that  
\begin{eqnarray}
\label{12fold2}
\sum_{n\ge0} m! S_2(n,m)\frac{x^n}{n!}=(e^x-1)^m.
\end{eqnarray}
\end{proof}

To prepare the construction of the ``symbolic local germ'' around zero, let us set, in the same manner as in \cite{GHM22,BHN,CM},
\begin{eqnarray}
\Dom_R(\Li)&:=&\{S\in\ncs{\C}{X}x_1\oplus \C1_{X^*}\vert\cr
&&\sum_{n\ge0}\Li_{[S]_n} \mbox{is unconditionally convergent in $\calH(D_{<R})$}\}
\end{eqnarray}
and prove the following:

\begin{proposition}\label{lem2}
With the notations as above, we have: 
\begin{enumerate}
\item The map $]0,1]\to \ncs{\C}{X}$ given by $R\mapsto\Dom_R(\Li)$ is strictly decreasing
\item Each $\Dom_R(\Li)$ is a shuffle subalgebra of $\CXX$.
\end{enumerate}
\end{proposition}

\begin{proof}
\begin{enumerate}
\item It is straightforward that he map $R \longmapsto\Dom_R(\Li)$ is decreasing. Set now, with $x_1^+=x_1x_1^*=x_1^*-1$,
\begin{eqnarray*}
S(t)=\sum\limits_{m\ge0}t^m(x_1^+)^{\shuffle m}
\end{eqnarray*}
and let $[S]_n(t)$ be its homogeneous components, we have
\begin{eqnarray*}
\sum_{n\ge0}\Li_{[S]_n(t)}(z)=\frac{1-z}{1-(t+1)z}.
\end{eqnarray*}
For $0<R_1<R_2\leq 1$ it is straightforward that
\begin{eqnarray*}
\Dom_{R_2}(\Li)\subset\Dom_{R_1}(\Li).
\end{eqnarray*}
Let us prove that the inclusion is strict. 

Take $\abs{z}<1$ and let us, be it finite or infinite, evaluate the sum 
\begin{eqnarray*}
M(z)=\sum_{n\ge0}\abs{\Li_{[S]_n(t)}(z)}=\sum_{n\ge0}\scal{S(t)}{x_1^n}\abs{\Li_{x_1^n}(z)}
\end{eqnarray*}
then
\begin{eqnarray*}
M(z)&=&\sum_{n\ge0}\abs{S(t)}{x_1^n}\abs{\Li_{x_1^n}(z)}\cr 
&=&\sum_{n\ge0}\sum_{m\ge0}\abs{t^m(x_1^+)^{\shuffle m}}{x_1^n}\abs{\Li_{x_1^n}(z)}\cr
&=&\sum_{m\ge0}m!t^m\sum_{n\ge0}S_2(n,m)\frac{\abs{\Li_{x_1}(z)}^n}{n!}\cr
&\le&\sum_{m\ge0}m!t^m\sum_{n\ge0}S_2(n,m)\frac{\Li^n_{x_1}(\abs{z})}{n!},
\end{eqnarray*}
due to the fact that $\abs{\Li_{x_1}(z)}\le\Li_{x_1}(\abs{z})$ (Taylor series with positive coefficients).
Finally, in view of equation (\ref{12fold2}), we get, on the one hand, for $\abs{z}<(t+1)^{-1}$, 
\begin{eqnarray*}
M(z)\le\sum_{m\ge0}t^m(e^{\Li_{x_1}(\abs{z})}-1)^m 
=\sum_{m\ge0}t^m(\frac{\abs{z}}{1-\abs{z}})^m=\frac{1-\abs{z}}{1-(t+1)\abs{z}}.
\end{eqnarray*}
This proves that, for all $r\in ]0,\dfrac{1}{t+1}[$, 
\begin{eqnarray*}
\sum_{n\ge0}\absv{\Li_{[S]_n(t)}(z)}_r<+\infty.
\end{eqnarray*}
On the other hand, if $(t+1)^{-1}\le\abs{z}<1$, one has $M(|z|)=+\infty$,
and the preceding calculation shows that, with $t$ choosen such that
\begin{eqnarray*}
0\le\dfrac{1}{R_2}-1<t<\dfrac{1}{R_1}-1,
\end{eqnarray*}
we have $S(t)\in\Dom_{R_1}(\Li)$ but $S(t)\notin\Dom_{R_2}(\Li)$ whence, for $0<R_1<R_2\le1$, 
$\Dom_{R_2}(\Li) \subsetneq\Dom_{R_1}(\Li)$.

\item  One has (proofs as in \cite{GHM22})
\begin{enumerate}
\item $1_{X^*}\in\Dom_R(\Li)$ (because $1_{X^*}\in\ncp{\C}{X}$) and $\Li_{1_{X^*}}=1_{\calH(\Omega)}$.
\item Taking $S,T\in\Dom_R(\Li)$ we have, by absolute convergence, $S\shuffle T\in\Dom_R(\Li)$.
It is easily seen that $S\shuffle T\in\ncs{\C}{X}x_1\oplus \C1_{X^*} $ and, moreover, that 
$\Li_S\Li_T=\Li_{S\shuffle T}$\footnote{Proof by absolute convergence as in \cite{GHM22}.}. 
\end{enumerate}
\end{enumerate}
\end{proof} 

In Theorem \ref{PropDomR} bellow, we study, for series taken in $\ncs{\C}{X}x_1\oplus\C.1_{X^*}$, the correspondence 
$\Li_\bullet$ to some $\calH(D_{<R})$, first (point 1) establishes its surjectivity (in a certain sense) and then
(points 2 and 3) examine the relation between summability of the functions and that of their Taylor coefficients.
For that, let us begin with a very general lemma on sequences of Taylor series which adapts, for our needs,
the notion of \textit{normal families} \cite{PM}.

\begin{lemma}\label{TaylSeq}
Let $\tau=(a_{n,N})_{n,N\ge0}$ be a double sequence of complex numbers. Setting
\begin{eqnarray*}
I(\tau):=\{r\in]0,+\infty[\vert\sum_{n,N\ge0}|a_{n,N}r^N|<+\infty\},
\end{eqnarray*}
one has
\begin{enumerate} 
\item $I(\tau)$ is an interval of $]0,+\infty[$, it is not empty iff there exists $z_0\in\C\setminus\{0\}$ such that 
\begin{eqnarray}\label{NonZeroRad}
\sum_{n,N\ge0}|a_{n,N}z_0^N|<+\infty
\end{eqnarray}
In this case, we set $R(\tau):=\sup(I(\tau))$, one has 
\begin{enumerate}
\item For all $N$, the series $\sum\limits_{n\ge0}a_{n,N}$ converges absolutely (in $\C$). 
Let us note $a_N$ - with one subscript - its limit
\item For all $n$, the convergence radius of the Taylor series 
\begin{eqnarray*}
T_n(z)=\sum_{N\ge0}a_{n,N}z^N
\end{eqnarray*}
is at least $R(\tau)$ and $\sum\limits_{n\in\N}T_n$ is summable for the standard
topology of $\calH(D_{<R(\tau)})$ with sum $T(z)=\sum\limits_{n,N\ge0}a_{N}z^N$.
\end{enumerate}

\item Conversely, we suppose that it exists $R>0$ such that 
\begin{enumerate}
\item Each Taylor series $T_n(z)=\sum\limits_{N\ge0}a_{n,N} z^N$ converges in $\calH(D_{<R})$.
\item The series $\sum\limits_{n\in\N}T_n$ converges unconditionnally in $\calH(D_{<R})$.
\end{enumerate}
Then $I(\tau)\not=\emptyset$ and $R(\tau)\ge R$.
\end{enumerate}
\end{lemma}

\begin{proof}
\begin{enumerate}
\item The fact that $I(\tau)\subset]0,+\infty[$ is straightforward from the definition.
If it exists $z_0\in \C$ such that
\begin{eqnarray*}
\sum_{n,N\ge0}\abs{a_{n,N}z_0^N}<+\infty
\end{eqnarray*}
then, for all $r\in]0,|z_0|[$, we have 
\begin{eqnarray*}
\sum_{n,N\ge0}\abs{a_{n,N}r^N}
=\sum_{n,N\ge0}\abs{a_{n,N}z_0^N}\biggl(\frac{r}{\abs{z_0}}\biggr)^N
\le\sum_{n,N\ge0}|a_{n,N}z_0^N|<+\infty
\end{eqnarray*}
in particular $I(\tau)\not=\emptyset$ and it is an interval of $]0,+\infty[$ with lower bound zero.     
\begin{enumerate}
\item Take $r\in I(\tau)$ (hence $r\not=0$) and $N\in \N$, then we get the expected result as
\begin{eqnarray*}
r^N\sum_{n\ge0}\abs{a_{n,N}}=\sum_{n\ge0}\abs{a_{n,N}r^N}\le\sum_{n,N\ge0}\abs{a_{n,N}r^N}<+\infty.
\end{eqnarray*}
\item Again, take any $r\in I(\tau)$ and $n\in\N$, then 
\begin{eqnarray*}
\sum\limits_{N\ge0}\abs{a_{n,N}r^N}<+\infty
\end{eqnarray*}
which proves that $R(T_n)\ge r$, hence the result\footnote{For a Taylor series $T$,
we note $R(T)$ the radius of convergence of $T$.}. We also have 
\begin{eqnarray*}
\abs{\sum_{N\ge0}a_Nr^N}\le\sum_{N\ge0}r^N\abs{\sum_{n\ge0}a_{n,N}}\le\sum_{n,N\ge0}\abs{a_{n,N}r^N}<+\infty
\end{eqnarray*}
and this proves that $R(T)\ge r$, hence $R(T)\ge R(\tau)$.
\end{enumerate}
\item Let $0<r<r_1<R$ and consider the path $\gamma(t)=r_1e^{2i\pi t}$, we have
\begin{eqnarray*}
\abs{a_{n,N}}=\abs{\frac{1}{2i\pi}\int_\gamma\frac{T_n(z)}{z^{N+1}}dz}\le 
\frac{2\pi}{2\pi}\frac{r_1\absv{T_n}_{K}}{r_1^{N+1}}\le\frac{\absv{T_n}_K}{r_1^{N}} 
\end{eqnarray*} 
with $K=\gamma([0,2\pi])$, hence
\begin{eqnarray*}
\sum_{n,N\ge0}\abs{a_{n,N}r^N}\le\sum_{n,N\ge0}\abs{T_n}_K(\frac{r}{r_1})^N\le\frac{r_1}{r_1-r}\sum_{n\ge0}\absv{T_n}_K<+\infty.
\end{eqnarray*}
\end{enumerate}
\end{proof}
 
\begin{theorem}\label{PropDomR}
\begin{enumerate} 
\item Let $T(z)=\sum\limits_{N\ge0}a_Nz^N$ be a Taylor series \textit{i.e.} such that 
\begin{eqnarray*}
\limsup_{N\to+\infty}\abs{a_N}^{1/n}=B<+\infty,
\end{eqnarray*}
then the series 
\begin{eqnarray}\label{preim_series1}
S=\sum_{N\ge0}a_N(-(-x_1)^+)^{\shuffle N}
\end{eqnarray} 
is summable (see \cite{NRSA}) in $\ncs{\C}{X}$ (with sum in $\ncs{\C}{x_1}$),
$S\in\Dom_R(\Li)$ with $R=(B+1)^{-1}$ and $\Li_S=T$.

\item Let $S\in\Dom_{R}(\Li)$ and $S=\sum\limits_{\ge0}[S]_n$ (homogeneous decomposition), we 
define $N\longmapsto\H_{\pi_Y(S)}(N)$ by\footnote{This definition is compatible with the old one 
when $S$ is a polynomial.} 
\begin{eqnarray*}
\frac{\Li_S(z)}{1-z}=\sum_{N\ge0}\H_{\pi_{Y}(S)}(N)z^N.
\end{eqnarray*}
\item\label{absconv1} Moreover,
\begin{eqnarray}\label{abs_sum0}
\forall r\in]0,R[,&&\sum\limits_{n,N\ge0}\abs{\H_{\pi_Y([S]_n)}(N)r^N}<+\infty,
\end{eqnarray} 
and, for all $N \in \N$, the series (of complex numbers), 
$\sum\limits_{n\ge0}\H_{\pi_Y([S]_n)}(N)$ converges absolutely to $\H_{\pi_Y(S)}(N)$. 
\item\label{QinDomH}  Conversely, let $Q \in \ncs{\C}{Y}$ with $Q=\sum\limits_{n\ge0}Q_n$
(decomposition by weights), we suppose that it exists $r\in ]0,1]$ such that 
\begin{eqnarray}\label{abs_sum2}
\sum\limits_{n,N\ge0}\abs{\H_{Q_n}(N)r^N}<+\infty,
\end{eqnarray}
in particular, for all $N\in \N$, $\sum\limits_{n\ge0}\H_{Q_n}(N)=\ell(N)\in\C$
unconditionally. Under such circumstances, $\pi_X(Q)\in\Dom_r(\Li)$ and, for all $z\in\C,\abs{z}\le r$,
\begin{eqnarray}\label{L2H_corresp2}
\frac{\Li_S(z)}{1-z}=\sum_{N\ge0}\ell(N)z^N,
\end{eqnarray}
\end{enumerate}
\end{theorem}

\begin{proof}
\begin{enumerate}
\item The fact that the series (\ref{preim_series1}) is summable comes from the fact that 
\begin{eqnarray*}
\omega(a_N(-(-x_1)^+)^{\shuffle N})\ge N
\end{eqnarray*}
(see \cite{NRSA}). Now from the lemma, we get 
\begin{eqnarray*}
(S)_n=\sum_{N\ge0}(a_N(-(-x_1)^+)^{\shuffle N})_n=(-1)^{N+n}a_NN!S_2(n,N)x_1^n.
\end{eqnarray*}
Then, with $r=\sup_{z\in K}\abs{z}$ (we have indeed $r=\absv{Id}_K$) and
taking into account that $\absv{\Li_{x_1}}_K\le\log({1}/(1-r))$, we have
\begin{eqnarray*}
\sum_{n\ge0}\absv{\Li_{(S)_n}}_K
&\le&\sum_{n\ge0}\sum_{N\ge0}\abs{a_N}N!S_2(n,N)\absv{\Li_{x_1^n}}_K\cr
&\le&\sum_{n\ge0}\sum_{N\ge0}\abs{a_N}N!S_2(n,N)\frac{\absv{\Li_{x_1}}_K^n}{n!} \cr 
&\le&\sum_{N\ge0}\abs{a_N}\sum_{n\ge0}N!S_2(n,N)\frac{\abs{\Li_{x_1}}_K^n}{n!} \cr
&\le&\sum_{N\ge0}\abs{a_N}(e^{\log(\frac{1}{1-r})}-1)^N \cr
&=&\sum_{N\ge0}\abs{a_N}\biggl(\frac{r}{1-r}\biggr)^N.
\end{eqnarray*}
Now if we suppose that $r\le(B+1)^{-1}$, we have $r(1-r)^{-1}\le{1}/{B}$
and this shows that the last sum is finite.

\item This point and next point are consequences of Lemma \ref{TaylSeq}.

Now, considering the homogeneous decomposition
\begin{eqnarray*}
S=\sum\limits_{n\ge0}[S]_n\in\Dom_R(\Li).
\end{eqnarray*}
we first establish inequation \eqref{abs_sum0}. Let $0<r<r_1<R$ and consider the path 
$\gamma(t)=r_1e^{2i\pi t}$, we have
\begin{eqnarray*}
\abs{\H_{\pi_Y([S]_n)}(N)}=\abs{\frac{1}{2i\pi}\int_\gamma\frac{\Li_{[S]_n}(z)}{(1-z)z^{N+1}}dz}\le 
\frac{2\pi}{2\pi}\frac{\absv{\Li_{[S]_n}}_K}{(1-r_1)r_1^{N+1}},
\end{eqnarray*}
$K=\gamma([0,1])$ being the circle of center $0$ and radius $r_1$.
Taking into account that, for $K\subset_{comp.}D_{<R}$, we have a decomposition
\begin{eqnarray*}
\sum_{n\in\N}\abs{\Li_{[S]_n}}_K=M<+\infty,
\end{eqnarray*}
we get 
\begin{eqnarray*}
\sum_{n,N\ge0}\abs{\H_{\pi_Y([S]_n)}(N)r^N}
&=&\sum_{n,N\ge0}\abs{\H_{\pi_Y([S]_n)}(N)r_1^N}(\frac{r}{r_1})^N\cr 
&=&\sum_{N\ge0}(\frac{r}{r_1})^N\sum_{n\ge0}\abs{\H_{\pi_Y([S]_n)}(N)r_1^N}\cr
&\le&\sum_{N\ge0}(\frac{r}{r_1})^N\frac{M}{(1-r_1)r_1}\cr
&\le&\frac{M}{(1-r_1)(r_1-r)}<+\infty.
\end{eqnarray*}  
The series 
$\sum\limits_{n\ge0}\Li_{[S]_n}(z)$ converges to $\Li_{S}(z)$ in $\calH(D_{<R})$
($D_{<R}$ is the open disk defined by $|z|<R$). For any $N\ge0$, by Cauchy's formula, one has, 
\begin{eqnarray*}
\H_{\pi_Y(S)}(N)
&=&\frac{1}{2i\pi}\int_\gamma\frac{\Li_{S}(z)}{(1-z)z^{N+1}}dz\cr
&=&\frac{1}{2i\pi}\int_\gamma\frac{\sum_{n\ge0}\Li_{[S]_n}(z)}{(1-z)z^{N+1}}dz\cr
&=&\frac{1}{2i\pi}\sum_{n\ge0}\int_\gamma\frac{\Li_{[S]_n}(z)}{(1-z)z^{N+1}}dz\cr
&=&\sum_{n\ge0}\H_{\pi_Y([S]_n)}(N)
\end{eqnarray*}
the exchange of sum and integral being due to the compact convergence.
The absolute convergence comes from the fact that the convergence of
$\sum\limits_{n\ge}\Li_{[S]_n}(z)$ is unconditional \cite{Sch}.
\item Fixing $N\in\N$, from inequation \eqref{abs_sum2}, we get
$\sum\limits_{n\ge0}\abs{\H_{Q_n}(N)}<+\infty$ which proves the absolute convergence.
Remark now that $(\pi_X(Q))_n=\pi_X(Q_n)$ and $\pi_Y(\pi_X(Q_n))=Q_n$,
one has, for all $\abs{z}\le r$ 
\begin{eqnarray*}
\abs{\Li_{\pi_X(Q_n)}(z)}=\abs{\sum_{N\in\N}\H_{Q_n}(N)z^N}\le\abs{\sum_{N\in \N}\H_{Q_n}(N)r^N},
\end{eqnarray*}
in other words
\begin{eqnarray*}
\absv{\Li_{\pi_X(Q_n)}}_{D\leq r}\le\abs{\sum_{N\in \N}\H_{Q_n}(N)r^N}
\end{eqnarray*}
and 
\begin{eqnarray*}
\sum_{n\in \N}\absv{\Li_{\pi_X(Q_n)}}_{D\le r}\le\abs{\sum_{n,N\in\N}\H_{Q_n}(N)r^N}<+\infty
\end{eqnarray*}
which shows that $\pi_X(Q)\in\Dom_r(\Li)$. The equation (\ref{L2H_corresp2}) is a consequence of point 2, taking $S=\pi_X(Q)$.
\end{enumerate}
\end{proof}

\begin{definition}
We set 
\begin{eqnarray*}
\Dom^{\mathrm{loc}}(\Li)=\bigcup\limits_{0<R\le1}\Dom_R(\Li) ; 
\Dom(\H_\bullet)=\pi_Y(\Dom^{\mathrm{loc}}(\Li))
\end{eqnarray*}
and, for $S\in\Dom^{\mathrm{loc}}(\Li)$,
\begin{center}
$\Li_S(z)=\sum\limits_{n\ge0}\Li_{[S]_n}(z)$ and $\dfrac{\Li_S(z)}{1-z}=\sum\limits_{N\ge0}\H_{\pi_Y(S)}(N)z^N$.
\end{center}
\end{definition}

Observe that, from this definition, theorem \eqref{theomain}, will show that $\Dom(\H_\bullet)$ is a stuffle subalgebra 
of $\ncs{\C}{Y}$.  
\begin{enumerate}
\item The series $T=\sum\limits_{n=1}^{\infty}(-1)^{n-1}y_n/n\in\ncs{\C}{Y}$ is not in 
$\Dom(\H_\bullet)$ because, for all $0<r<1$, one has 
\begin{eqnarray}
\sum_{n,N}\abs{T_n(N)r^N}\ge\sum_{n\ge0}\frac{1}{1-r}=+\infty
\end{eqnarray}
However one can get unconditional convergence using a sommation by pairs (odd + even).
\item For all $s\in]1,+\infty[$, the series $T(s)=\sum\limits_{n=1}^{\infty}(-1)^{n-1}y_nn^{-s}\in\ncs{\C}{Y}$ is in $\Dom(\H_\bullet)$.
\end{enumerate}

We can now state the 
\begin{theorem}\label{theomain}
Let $S,T\in\Dom^{\mathrm{loc}}(\Li)$, then
\begin{eqnarray*}
S\shuffle T\in\Dom^{\mathrm{loc}}(\Li),\pi_X(\pi_Y(S)\stuffle\pi_Y(T))\in\Dom^{\mathrm{loc}}(\Li)
\end{eqnarray*}
and for all $N\ge0$,
\begin{eqnarray}
\Li_{S\shuffle T}&=&\Li_{S}\Li_{T};\quad\Li_{1_{X^*}}=1_{\calH(\Omega)},\label{eq1}\\
\H_{\pi_Y(S)\stuffle \pi_Y(T)}(N)&=& \H_{\pi_Y(S)}(N)\H_{\pi_Y(T)}(N).\label{eq2}\\
\dfrac{\Li_{S}(z)}{1-z}\odot\dfrac{\Li_{T}(z)}{1-z}&=&\dfrac{\Li_{\pi_X(\pi_Y(S)\stuffle\pi_Y(T))}(z)}{1-z}.\label{eq3}
\end{eqnarray}
\end{theorem}

\begin{proof}
For equation (\ref{eq1}), we get, from lemma \ref{lem2} that $Dom^{loc}(\Li)$ is the union of an increasing
set of shuffle subalgebras of $\ncs{\C}{X}$. It is therefore a shuffle subalgebra of the latter. 

For equation \eqref{eq2}, suppose $S\in\Dom^{R_1}_0(\Li)$ (resp. $T\in\Dom^{R_2}_0(\Li)$).
By \cite{Had_prod} and theorem \ref{PropDomR}, one has 
\begin{eqnarray*}
\dfrac{\Li_{S}(z)}{1-z}\odot\dfrac{\Li_{T}(z)}{1-z}\in\Dom^{R_1R_2}_0(\Li),
\end{eqnarray*}
where $\odot$ stands for the Hadamard product \cite{Had_prod}. Hence, for $|z|<R_1R_2$, one has
\begin{eqnarray}
f(z)=\dfrac{\Li_{S}(z)}{1-z}\odot \dfrac{\Li_{T}(z)}{1-z}=\sum_{N\ge0}\H_{\pi_Y(S)}(N)\H_{\pi_Y(T)}(N)z^N
\end{eqnarray}
and, due to theorem \ref{PropDomR} point \eqref{absconv1}, for all $N$,
$\sum\limits_{p\ge0}\H_{\pi_Y(S_p)}(N)=\H_{\pi_Y(S)}(N)$ and
$\sum_{q\ge0}\H_{\pi_Y(T_q)}(N)=\H_{\pi_Y(T)}(N)$ (absolute convergence)
then, as the product of two absolutely convergent series is absolutely convergent
(w.r.t. the Cauchy product), one has, for all $N$,
\begin{eqnarray}
\H_{\pi_Y(S)}(N)\H_{\pi_Y(T)}(N)
&=&\biggl(\sum_{p\ge0}\H_{\pi_Y(S_p)}(N)\biggr)\biggl(\sum_{q\ge0}\H_{\pi_Y(T_q)}(N)\biggr)\cr 
&=&\sum_{p,q\ge0}\H_{\pi_Y(S_p)}(N)\H_{\pi_Y(T_q)}(N)\cr
&=&\sum_{n\ge0}\sum_{p+q=n}\H_{\pi_Y(S_p)\stuffle\pi_Y(T_q)}(N)\cr
&=&\sum_{n\ge0}\H_{(\pi_Y(S)\stuffle\pi_Y(T))_n}(N). 
\end{eqnarray}

Remains to prove that condition of Theorem \ref{PropDomR}, \textit{i.e.} inequation \eqref{abs_sum2}
is fulfilled. To this end, we use the well-known fact that if $\sum_{m\ge 0}c_m\,z^m$ has radius of
convergence $R>0$, then $\sum\limits_{m\ge 0}\abs{c_m}z^m$ has the same radius of convergence
(use $1/R=\limsup_{m\ge 1}\abs{c_m}^{-m}$), then from the fact that $S\in\Dom^{R_1}_0(\Li)$ 
(resp. $T\in Dom^{R_2}_0(\Li)$), we have \eqref{abs_sum0} for each of them and,
using the Hadamard product of these expressions, we get
\begin{eqnarray*}
\forall r\in]0,R_1.R_2[,&&\sum_{p,q,N\ge0}|\H_{\pi_Y(S_p)}(N)\H_{\pi_Y(T_q)}(N)\,r^N|<+\infty,
\end{eqnarray*}
and this assures, for $|z|<R_1R_2$, the convergence of 
\begin{eqnarray}
f(z)=\sum_{n,N\ge0}\H_{(\pi_Y(S)\stuffle\pi_Y(T))_n}(N)z^N
\end{eqnarray}
applying theorem \ref{PropDomR} point \eqref{QinDomH} to $Q=\pi_Y(S)\stuffle\pi_Y(T)$
(with any $r<R_1R_2$), we get $\pi_X(Q)=\pi_X(\pi_Y(S)\stuffle\pi_Y(T))\in\Dom^{\mathrm{loc}}(\Li)$ and 
\begin{eqnarray*}
f(z)=\sum_{N\ge0}\Big(\sum_{n\ge0}\H_{(\pi_Y(S)\stuffle\pi_Y(T))_n}(N)\Big)z^N
=\dfrac{\Li_{\pi_X(\pi_Y(S)\stuffle\pi_Y(T))}(z)}{1-z}.
\end{eqnarray*}
hence \eqref{eq2}.
\end{proof}

\subsection{Stuffle product and stuffle characters}\label{stuffledef}

For the some reader's convenience, we recall here the definitions of shuffle and stuffle products.
As regards shuffle, the alphabet $\calX$ is arbitrary and $\shuffle$ is defined by the following recursion 
(for $a,b\in\calX$ and $u,v\in\calX^*$)  
\begin{eqnarray}
u\shuffle 1_{\calX^*}&=&1_{\calX^*}\shuffle u=u,\\
au\shuffle bv&=& a(u\shuffle bv)+b(au\shuffle v).
\end{eqnarray}
As regards stuffle, the alphabet is $Y=Y_{\N_{\ge1}}=\{y_s\}_{s\in\N_{\ge1}}$ and $\stuffle$ is defined by the following recursion
\begin{eqnarray}
u\stuffle 1_{Y^*}&=&1_{Y^*}\stuffle u=u,\\
y_su\stuffle y_tv&=&y_s(u\stuffle y_tv)+y_t(y_su\stuffle v) + y_{s+t}(u\stuffle v).
\end{eqnarray}
Be it for stuffle or shuffle, the noncommutative\footnote{For concatenation.} polynomials equipped with this product form an
associative commutative and unital algebra namely $(\ncp{\C}{X},\shuffle,\allowbreak1_{X^*})$ (resp. $(\ncp{\C}{Y},\stuffle,1_{Y^*})$). 

\begin{example}
As examples of characters, we have already seen
\begin{itemize}
\item $\Li_{\bullet}$ from $(\Dom^{\mathrm{loc}}(\Li_{\bullet}),\shuffle,1_{X^*})$ to $\calH(\Omega)$
\item $\H_{\bullet}$ from $(\Dom(\H_{\bullet}),\stuffle,1_{Y^*})$ to $\C^{\N}$ (arithmetic functions $\N\longrightarrow\C$)
\end{itemize}
\end{example}
In general, a character from a $k$-algebra\footnote{Here we will use $k=\Q$ or $\C$.} $(\calA,*_1,1_{\calA})$ with
values in $(\calB,*_2,1_{\calB})$ is none other than a morphism between the $k$-algebras $\calA$ and a commutative
algebra\footnote{In this context all algebras are associative and unital.} $\calB$. The algebra $(\calA,*_1,1_{\calA})$
does not have to be commutative for example characters of  $(\ncp{\C}{\calX},\conc,1_{\calX^*})$
- \textit{i.e.} $\conc$-characters - where all proved to be of the form \cite{PVNC}
\begin{eqnarray}\label{Plane}
\biggl(\sum_{x\in\calX}\alpha_xx\biggr)^*
\end{eqnarray}
\textit{i.e.} Kleene stars of the plane \cite{GHM22,BHN,CM}.
They are closed under shuffle and stuffle and endowed with these laws,
they form a group. Expressions like \eqref{Plane} (\textit{i.e.}
homogeneous series of degree 1) form a vector space noted $\widehat{\C.Y}$.

As a consequence, given $P=\sum\limits_{i\ge1}\alpha_iy_i$ and $Q=\sum\limits_{j\ge1}\beta_jy_j$,
we know in advance that their stuffle is a $\tt conc$-character \textit{i.e.} of the form
$(\sum\limits_{n\ge1}c_ny_n)^*$. Examining the effect of this stuffle on each letter (which suffices),
we get the identity \cite{PVNC}
\begin{eqnarray}\label{WI}
\biggl(\sum_{i\ge1}\alpha_iy_i\biggr)^*\stuffle\biggl(\sum_{j\ge1}\beta_jy_j\biggr)^*=
\biggl(\sum_{i\ge1}\alpha_iy_i+\sum_{j\ge1}\beta_jy_j+\sum_{i,j\ge1}\alpha_i\beta_jy_{i+j}\biggr)^*
\end{eqnarray}  
which suggests to take an auxiliary variable, say $q$, and code ``the plane'' $\widehat{\C.Y}$,
\textit{i.e.} expressions like \eqref{Plane}, like in Umbral calculus by

\begin{eqnarray}\label{Umbra}
\pi_Y^{\mathrm{Umbra}}:\sum_{n\ge1}\alpha_n\,q^n\longmapsto\sum_{n\ge1}\alpha_ny_n   
\end{eqnarray}
which is linear and bijective\footnote{Its inverse will be naturally noted $\pi_q^{\mathrm{Umbra}}$.} from $\C_+[[q]]$ to $\widehat{\C.Y}$.

With this coding at hand and for $S,T\in \C_+[[q]]$, identity \eqref{WI} reads 
\begin{eqnarray}\label{compo1}
(\pi_Y^{\mathrm{Umbra}}(S))^*\stuffle (\pi_Y^{\mathrm{Umbra}}(T))^*=(\pi_Y^{\mathrm{Umbra}}((1+S)(1+T)-1))^*
\end{eqnarray}  
This shows that if one sets, for $z\in\C$ and $T\in\C_+[[x]]$,  
\begin{eqnarray}
G(z)=(\pi_Y^{\mathrm{Umbra}}(e^{zT}-1))^* 
\end{eqnarray}
we get a one-parameter stuffle group\footnote{\textit{i.e.} $G(z_1+z_2)=G(z_1)\stuffle G(z_2);G(0)=1_{Y^*}$.}, drawn on 
$1+\ncs{\C[z]_+}{Y}$ (a Magnus group), i.e. such that every 
coefficient is polynomial in $z$. Differentiating it we get 
\begin{eqnarray}\label{diffeq3}
\frac{d}{dz}(G(z))=(\pi_Y^{\mathrm{Umbra}}(T))G(z)
\end{eqnarray}
and \eqref{diffeq3} with the initial condition $G(0)=1_{Y^*}$ integrates as 
\begin{eqnarray}\label{soldiffeq3}
G(z)=\exp_{\stuffle}(z\pi_Y^{\mathrm{Umbra}}(T))
\end{eqnarray}
where the exponential map for the stuffle product is defined,
for any $P\in\ncs{\C}{Y}$ such that $\scal{P}{1_{Y^*}}=0$, is defined by 
\begin{eqnarray}
\exp_{\stuffle}(P):=1_{Y^*}+\dfrac{P}{1!}+\dfrac{P\stuffle P}{2!}+\ldots+\dfrac{P^{\stuffle n}}{n!}+\ldots.
\end{eqnarray}
In particular, from \eqref{soldiffeq3}, one gets, for $k\ge1$,
\begin{eqnarray}\label{WIk}
(zy_k)^*&=&\exp_{\stuffle}\biggl(-\sum_{n\ge1}y_{nk}\frac{(-z)^n}{n}\biggr).
\end{eqnarray}
This expression and that of 
\begin{eqnarray}
\frac1{\Gamma(1+z)}=\exp\biggl(\gamma z-\sum_{n\ge2}\zeta(n)\frac{(-z)^n}n\biggr)
\end{eqnarray} 
suggests to consider lacunary analogues of the inverse Gamma function together with a character which sends 
$y_1$ to $\gamma$ and $y_n,n\ge2$ to $\zeta(n)$. This $\stuffle$-character is provided by asymptotic analysis
of the Harmonic Sums. Indeed, one can show that, $w\in Y^*$ being given, the asymptotic expansion of $N\longmapsto\H_w(N)$,
along the asymptotic scale $(\log(N)^pN^{-q})_{p,q\in \N}$, at any rate\footnote{This means that the
following expression is the limit of all partial asymptotic expansions.}, can be written 
\begin{eqnarray}\label{AsympH}
\sum_{q\ge0}Q_{w,q}(\log(N))N^{-q},
\end{eqnarray}
where $Q_{w,q}\in\C[X]$ (univariate polynomials) and, in particular, $Q_{w,0}\in\Q[\gamma][X]$ \cite{Daresbury}. 
From this and the fact that $\H_\bullet$ is a $\stuffle$-character, one gets that $w\longmapsto Q_{w,q}$ (resp. 
$\gamma_\bullet:w\longmapsto Q_{w,q}(0)$) is a $\stuffle$-character with values in $\Q[\gamma][X]$ (resp. $\Q[\gamma][X]$). 

Now, a domain\footnote{Open, nonempty and connected subset of $\C$.} $\Omega$ being given, it is easy to see that any  
$\stuffle$-character $\chi$ (with general complex values and in particular $\gamma_\bullet$) classically extends  
$\ncp{\calH(\Omega)}{Y}$ by 
\begin{eqnarray}
\chi(P)=\sum\limits_{w\in Y^*}\scal{P}{w}\scal{\chi}{w}     
\end{eqnarray}
as a $\stuffle$-character from $\ncp{\calH(\Omega)}{Y}$ with values in $\calH(\Omega)$. 

Now, we can extend $\chi$ to some series, over $Y$. For that, let us set as above and as in \cite{GHM22,BHN,CM},
\begin{definition}
For any $T\in \ncs{\calH(\Omega)}{Y}$, we note $[T]_n$ the homogeneous component\footnote{
The weight (w) of $w \in Y^*$ is just the sum of its indices} $\sum\limits_{|w|=n}\scal{T}{w}w$ of $T$ 
\begin{eqnarray}
\Dom(\chi,\Omega)=\{T\in\ncs{\calH(\Omega)}{Y}|(\chi(T_n))_{n\in\N}&\mbox{is summable in}&\calH(\Omega)\}
\end{eqnarray}
The result, $\sum\limits_{n\ge0}\chi(T_n)$ will be noted $\hat{\chi}(T)$.
\end{definition}
This being defined, we have the following theorem 

\begin{theorem}\label{stability1}
Let $\chi:\ncp{\C}{Y}\longrightarrow\C$ be a $\stuffle$-character\footnote{We will still note its extension to $\ncp{\calH(\Omega)}{Y}$ by $\chi$.}
\begin{enumerate}
\item $\ncp{\calH(\Omega)}{Y}\subset \Dom(\chi,\Omega)$  
\item If $S,T\in\Dom(\chi,\Omega)$ then $S\stuffle T\in\Dom(\chi,\Omega)$\footnote{In fact $\Dom(\chi,\Omega)$ is a 
subalgebra of $(\ncs{\calH(\Omega)}{Y},\stuffle,1_{Y^*})$} and 
\begin{eqnarray}\label{shufflechar}
\hat{\chi}(S\stuffle T)=\hat{\chi}(S)\hat{\chi}(T)
\end{eqnarray}
\item If $S\in\Dom(\chi,\Omega)$, then $\exp_{\stuffle}(S)\in \Dom(\chi,\Omega)$ and 
\begin{eqnarray*}
\hat{\chi}(\exp_{\stuffle}(S))=e^{\hat{\chi}(S)}
\end{eqnarray*}
\end{enumerate}
\end{theorem}

\begin{proof}
\begin{enumerate}
\item By finitely supported sum.
\item $S,T\in\Dom(\chi,\Omega)$ then $(\chi(S_p))_{p\ge0},(\chi(T_q))_{p\ge0}$ are summable.
But, as $\stuffle$ is graded for the weight, one has $[S\stuffle T]_n=\sum_{p+q=n}[S]_p\stuffle [T]_q$.
Take any $K$ nonempty compact within $\Omega$, then 
\begin{eqnarray*}
\sum_{n\ge0}\absv{\chi([S\stuffle T]_n)}_K
&=&\sum_{n\ge0}\absv{\chi\biggl(\sum_{p+q=n}[S]_p\stuffle [T]_q\biggr)}_K\cr
&=&\sum_{n\ge0}\absv{\sum_{p+q=n}\chi([S]_p)\chi([T]_q)}_K\cr 
&\le&\sum_{n\ge0}\absv{\sum_{p+q=n}\chi([S]_p)}_K\absv{\chi([T]_q)}_K\cr
&=&\sum_{p,q\ge0}\absv{\chi([S]_p)}_K\absv{\chi([T]_q)}_K<+\infty.
\end{eqnarray*}
The same computation without the seminorm proves \eqref{shufflechar}.

\item If $S\in \Dom(\chi,\Omega)$ and we have to examine (and prove)
the summability of the family $(\chi([\exp_{\stuffle}(S)]_n))_{n\ge0}$.
Setting $S=\sum_{q\ge0}[S]_q$, we have 
\begin{eqnarray}
[\exp_{\stuffle}(S)]_n=\sum_{m\ge0}\sum_{q_1+2q_2+\cdots mq_m=n}
\frac{[S]_1^{\stuffle q_1}\stuffle\cdots\stuffle[S]_m^{\stuffle q_m}}{q_1!q_2!\cdots q_m!}.
\end{eqnarray}
Hence, with all $q_i>0$,
\begin{eqnarray}
\sum_{n\ge0}\absv{\chi([\exp_{\stuffle}(S)]_n)}_K
&\le&1+\sum_{n>0}\sum_{m>0}\sum_{q_1+2q_2+\cdots mq_m=n}\\
&&\hfill\frac{\absv{\chi([S]_1}_K^{q_1}\cdots\absv{\chi([S]_m)}_K^{q_m}}{q_1!q_2!\cdots q_m!}\cr
&\le&\prod_{q\ge1}e^{\absv{\chi([S]_q}_K}\cr 
&=&e^{\sum_{q\ge1}\absv{\chi([S]_q}_K}\cr 
&=&e^M<+\infty.
\end{eqnarray}
because, as $S\in\Dom(\chi,\Omega)$, we have $\sum_{q\ge1}\absv{\chi([S]_q)}_K=M<+\infty$.
\end{enumerate}
\end{proof}

\subsection{A remarkable set of exponents}
On the formal side, from \eqref{WIk}, we have \cite{PVNC}
\begin{eqnarray}
(z^ky_k)^*=\exp_{\stuffle}\biggl(-\sum_{n\ge1}y_{nk}\frac{(-z)^{nk}}{n}\biggr),
&\mbox{for}&z\in\C,\abs{z}<1,
\end{eqnarray}
and transform it through the $\stuffle$-character $\hat\gamma_\bullet$. First of all, we compute the radius
of convergence of the image of the exponent (for coherence  with the ``bullet-notation'', we will note
$\gamma_{y_n}$ the image of $y_n$ by the character $\gamma_\bullet$) which gives \cite{BHN}
\begin{eqnarray}\label{ell}
\mbox{for }z\in\C,\abs{z}<1,&&
\ell_k(z)=\left\{
\begin{array}{rclcl}
\gamma z-\displaystyle\sum_{n\ge2}\zeta(n)\frac{(-z)^n}{n}&\mbox{if}&k=1,\cr
-\displaystyle\sum_{n\ge1}\zeta(nk)\frac{(-z)^{nk}}{n}&\mbox{if}&k>1.
\end{array}\right.
\end{eqnarray}
Then, from the fact that $1<\zeta(n)\le\zeta(2)=\pi^2/6$ (for $n\ge2$),
we get that the radius of convergence of all $\ell_k(z)$ is $R=1$. Therefore,
we set $\Omega=D_{<1}$, the open disk of radius one centered at zero and get
that  all $-\sum\limits_{n\ge1}y_{nk}{(-z)^{nk}}/n$ belong to
$\Dom(\hat\gamma_{\bullet},\Omega)$. Third point of theorem \eqref{stability1}
implies at once 
\begin{itemize}
\item Their exponentials \cite{BHN,PVNC}
\begin{eqnarray}
(z^ky_k)^*=\exp_{\stuffle}\biggl(-\sum_{n\ge1}y_{nk}\frac{(-z)^{nk}}{n}\biggr);
&\mbox{for}&z\in\C,\abs{z}<1.
\end{eqnarray}
are all in $\Dom(\hat\gamma_\bullet,\Omega)$ and therefore linearly independent.
\item and their transforms through $\hat\gamma_\bullet$ follow exponentiation (for $z\in\C,\abs{z}<1$), \textit{i.e.} \cite{BHN}
\begin{eqnarray}
\hat\gamma_{(zy_k)^*}=\exp(\ell_k(z))=\left\{
\begin{array}{rclcl}
\exp\biggl(\gamma z-\displaystyle\sum_{n\ge2}\frac{(-z)^n\zeta(n)}{n}\biggr)&\mbox{if}&k=1,\cr
\exp\biggl(-\displaystyle\sum_{n\ge1}\zeta(nk)\frac{(-z)^{nk}}{n}\biggr)&\mbox{if}&k>1.
\end{array}\right.
\end{eqnarray}
\end{itemize}   
This leads us to set, for all $k\ge1$ and for $z\in\C,\abs{z}<1$, \cite{BHN}
\begin{eqnarray}\label{NewEulerian}
\Gamma_{y_k}(1+z):=e^{-\ell_k(z)},&\mbox{for}&z\in\C,\abs{z}<1.
\end{eqnarray}

\begin{proposition}[\cite{BHN}]\label{plein}
The families $(\ell_r)_{r\ge1}$ and $(e^{\ell_r})_{r\ge1}$
are $\C$-linearly free and free from $1_{\calH(\Omega)}$.
\end{proposition}

\begin{proof}
Since $(\ell_r)_{r\ge1}$ is triangular\footnote{A family $(g_i)_{i\ge1}$
is said to be {\it triangular} if the valuation of $g_i,\varpi(g_i),$ equals $i\ge1$.
It is easy to check that such a family is $\C$-linearly free and that is also the case
of families such that $(g_i-g(0))_{i\ge1}$ is triangular.} then $(\ell_r)_{r\ge1}$ is
$\C$-linearly free. So is $(e^{\ell_r}-e^{\ell_r(0)})_{r\ge1}$, being triangular,
we get that $(e^{\ell_r})_{r\ge1}$ is $\C$-linearly independent and free from $1_{\calH(\Omega)}$.
\end{proof}

Now, for any $r\ge1$, let $G_r$ (resp. $\calG_r$) denote the set (resp. group) of solutions,
$\{\xi_0,\ldots,\xi_{r-1}\}$, of the equation $z^r=(-1)^{r-1}$ (resp. $z^r=1$). If $r$ is odd,
it is a group as $G_r=\calG_r$ otherwise it is an orbit as $G_r=\xi\calG_r$, where $\xi$ is
any solution of $\xi^r=-1$ (this is equivalent to $\xi\in\calG_{2r}$ and $\xi\notin\calG_r$).
For $r,q\ge1$, we will need also a system $\mathbb{X}$ of representatives of $\calG_{qr}/\calG_r$,
{\it i.e.} $\mathbb{X}\subset \calG_{qr}$ such that 
\begin{eqnarray}
\calG_{qr}=\biguplus_{\tau\in\mathbb{X}}\tau\calG_r.
\end{eqnarray}
It can also be assumed that $1\in\mathbb{X}$ as with $\mathbb{X}=\{e^{2\mathrm{i}k\pi/qr}\}_{0\le k\le q-1}$.

\begin{proposition}[\cite{BHN}]\label{Weierstrass2}
\begin{enumerate}
\item For $r\ge1,\chi\in\calG_r$ and $z\in\C,|z|<1$,
the functions $\ell_r$ and $e^{\ell_r}$ have the symmetry,
$\ell_r(z)=\ell_r(\chi z)$ and $e^{\ell_r(z)}=e^{\ell_r(\chi z)}$.

In particular, for $r$ even, as $-1\in\calG_r$, these functions are even.
\item For $|z|<1$, we have
\begin{eqnarray*}
\ell_r(z)=-\sum_{\chi\in G_r}\log(\Gamma(1+\chi z))&\mbox{and}&
e^{\ell_r(z)}=\prod_{\chi\in G_r}e^{\gamma\chi z}
\prod_{n\ge1}(1+{\chi z}/{n})e^{-{\chi z}/n}.
\end{eqnarray*}
\item For any odd $r\ge2$,
\begin{eqnarray*}
\Gamma_{y_r}^{-1}(1+z)=e^{\ell_r(z)}=\Gamma^{-1}(1+z)\prod_{\chi\in G_r\rsetminus{\{1\}}}e^{\ell_1(\chi z)}
\end{eqnarray*}
\item and, in general, for any odd or even $r\ge2$,
\begin{eqnarray*}
\ell_r(z)=\prod_{\chi\in G_r}e^{\ell_1(\chi z)}=\prod_{n\ge1}(1+{z^r}/{n^r}).
\end{eqnarray*}
\item For $r\ge1$, the function $\ell_r$ is holomorphic on the open unit disc, $D_{<1}$,
\item For $r\ge1$, the  function $e^{\ell_r}$ (resp. $e^{-\ell_r}$) is entire
(resp. meromorphic), and admits a countable set of isolated zeroes (resp. poles)
on the complex plane which is expressed as $\biguplus_{\chi\in G_r}\chi\Z_{\le-1}$.
\end{enumerate}
\end{proposition}

\begin{proof}
The results are known for $r=1$ ({\it i.e.} for $\Gamma^{-1}$). For $r\ge2$, we get
\begin{enumerate}
\item By \eqref{ell}, with $\chi\in \calG_r$, we get
\begin{eqnarray*}
\ell_r(\chi z)=-\sum_{n\ge1}\zeta(kr)\frac{(-\chi^rz^r)^k}k=-\sum_{k\ge1}\zeta(kr)\frac{(-z^r)^k}k=\ell_r(z),
\end{eqnarray*}
thanks to the fact that, for any $\chi\in\calG_r$, one has $\chi^r=1$.

In particular, if $r$ is even then $\ell_r(z)=\ell_r(-z)$, {\it i.e.} $\ell_r$ is even.
\item If $r$ is odd, as $G_r=\calG_r$ and, applying the symmetrization principle\footnote{
Within the same disk of convergence as $f$, one has,
\begin{eqnarray*}
f(z)=\sum_{n\ge1}a_nz^n&\mbox{and}&
\sum_{\chi\in\calG_r}f(\chi z)=r\sum_{k\ge1}a_{rk}z^{rk}.
\end{eqnarray*}}, we get 
\begin{eqnarray*}
-\sum_{\chi\in G_r}\ell_1(\chi z)
=-\sum_{\chi\in\calG_r}\ell_1(\chi z)
=r\sum_{k\ge1}\zeta(kr)\frac{(-z)^{kr}}{kr}
=\sum_{k\ge1}\zeta(kr)\frac{(-z^r)^k}{k}.
\end{eqnarray*}
The last term being due to the fact that, precisely, $r$ is odd.

If $r$ is even, we have the orbit $G_r=\xi\calG_r$ (still with $\xi^r=-1$) and then, by the same principle,
\begin{eqnarray*}
-\sum_{\chi\in\calG_r}\ell_1(\chi\xi z)
=r\sum_{k\ge1}\zeta(kr)\frac{(-\xi z)^{kr}}{kr}
=\sum_{k\ge1}\zeta(kr)\frac{\bigl((-\xi z)^{r}\bigr)^k}{k}
=\sum_{k\ge1}\zeta(kr)\frac{\bigl(-z^{r}\bigr)^k}{k}.
\end{eqnarray*}

\item Straightforward.

\item Due to the fact that the external product is finite, we can distribute it on each factor and get
\begin{eqnarray*}
e^{\ell_r(z)}=\overbrace{\Bigl(\prod_{\chi\in\calG_r}e^{\gamma\chi z}\Bigr)}^{=1}
\prod_{\matop{n\ge1}{\chi\in\calG_r}}\Bigl(1+\frac{\chi z}{n}\Bigr)e^{-\frac{\chi z}n}
=\overbrace{\Bigl(\prod_{\matop{n\ge1}{\chi\in\calG_r}}e^{-\frac{\chi z}n}\Bigr)}^{=1}
\prod_{\matop{n\ge1}{\chi\in\calG_r}}\Bigl(1+\frac{\chi z}{n}\Bigr).
\end{eqnarray*}
Using the elementary symmetric functions of $G_r$, we get the expected result.
\item One has $e^{\ell_1(z)}=\Gamma^{-1}(1+z)$ which proves the claim for $r=1$.
For $r\ge2$, note that $1\leq\zeta(r)\le\zeta(2)$ which implies that the radius of convergence
of the exponent is $1$ and means that $\ell_r$ is holomorphic on the open unit disc.
This proves the claim.
\item The function $e^{\ell_r(z)}=\Gamma_{y_r}^{-1}(1+z)$ (resp. $e^{-\ell_r(z)}=\Gamma_{y_r}(1+z)$)
is entire (resp. meromorphic) as finite product of entire (resp. meromorphic) functions, for $r\ge1$.
The factorization in Proposition \ref{Weierstrass2} yields the set of zeroes (resp. poles).
\end{enumerate}
\end{proof}

As an example of projection, through $\hat\gamma_\bullet$ of an algebraic identity
let us mention, for any $z\in\C,\abs{z}<1$, one has, from \eqref{WIk}, \cite{PVNC}
\begin{eqnarray}
(z^ky_k)^*\stuffle(-z^ky_k)^*=(-z^{2k}y_{2k})^*,
\end{eqnarray}
which, transformed by $\hat\gamma_\bullet$ and for $\abs{|z}<1$ and $k\ge1$,
amounts to Euler's reflection formula (generalized to arbitrary $k$) \cite{BHN}
\begin{eqnarray}\label{reflection}
\Gamma_{y_{2r}}(1+z)=\Gamma_{y_r}(1+\rho z)\Gamma_{y_r}(1+\rho\xi z),
\end{eqnarray} 
where $\rho$ is a $2r^{th}-$root of $(-1)$ and $\xi$ a primitive $2r^{th}$ root of unity.

It is well known that the function $e^{\ell_1(z)}=\Gamma^{-1}(1+z)$ is entire.
In fact, all functions  \eqref{NewEulerian} are so (see Proposition \ref{plein}). 
From this, ones get that \eqref{reflection} holds on the whole plane.

\begin{example}[\cite{words03,orlando}]
Let us give examples relating to polyzetas. For that, we use the following identities,
for $z\in\C,\abs{z}<1$, (see \cite{PVNC})
\begin{eqnarray*}
(-zx_1)^*\shuffle(zx_1)^*=1&\mbox{and}&(-zy_1)^*)^*\stuffle(zy_1)^*=(-z^2y_2)^*,\\
(-z^2x_0x_1)^*\shuffle(z^2x_0x_1)^*=(-z^4x_0^2x^2_1)^*&\mbox{and}&(-z^2y_2)^*\stuffle(z^2y_2)^*=(-z^4y_4)^*.
\end{eqnarray*}

\begin{itemize}
\item From \eqref{reflection} and for $r=1$, we have
\begin{eqnarray*}
\hat\gamma_{(-z^2y_2)^*}&=&\hat\gamma_{(zy_1)^*}\hat{\gamma}_{(-zy_1)^*}\\
\Gamma_{y_2}^{-1}(1+\mathrm{i}z)&=&\Gamma_{y_1}^{-1}(1+z)\Gamma_{y_1}^{-1}(1-z)\\
e^{-\sum\limits_{n\ge2}\zeta(2n){z^{2n}}/n}&=&\dfrac{\sin(z\pi)}{z\pi}
=\sum\limits_{k\ge1}\dfrac{(z\mathrm{i}\pi)^{2k}}{(2k)!}.
\end{eqnarray*}
One can show (with a suitable extension of $\zeta$, see \cite{words03,orlando}) that\footnote{
Recall that, for any $w\in Y^*\setminus y_1Y^*$, one has $\gamma_w=\zeta(\pi_X(w))$ \cite{CM}
and then it can be extended over series.} $\hat\gamma_{(-z^2y_2)^*}=\zeta((-z^2x_0x_1)^*)$.
Then, identifying the coeffients of $z^{2k}$, we get
\begin{eqnarray*}
\frac{\zeta(\overbrace{{2,\ldots,2}}^{k{\tt times}})}{\pi^{2k}}=\dfrac{1}{(2k+1)!}\in\Q.
\end{eqnarray*}

\item Now, with $r=2$, letting $\rho^4=-1$, we have
\begin{eqnarray*}
\hat\gamma_{(-z^4y_4)^*}&=&\hat\gamma_{(z^2y_2)^*}\hat{\gamma}_{(-z^2y_2)^*}\\
\Gamma_{y_4}^{-1}(1+ z)&=&\Gamma_{y_2}^{-1}(1+\rho z)\Gamma_{y_2}^{-1}(1+\mathrm{i}\rho z),\cr
e^{-\sum\limits_{k\ge1}\zeta(4k){(-1)^kz^{4k}}/k}
&=&\dfrac{\sin(\mathrm{i}\rho z\pi)}{\mathrm{i}\rho z\pi}\dfrac{\sin(\rho \pi z)}{z\pi \rho}.
\end{eqnarray*}
Again, with a suitable extension of\footnote{\textit{idem}.} $\zeta$ (see \cite{words03,orlando})
\begin{eqnarray*}
\hat{\gamma}_{(-z^4y_4)^*}=\zeta((-z^4y_4)^*),
&\hat{\gamma}_{(-z^2y_2)^*}=\zeta((-z^2y_2)^*),
&\hat{\gamma}_{(z^2y_2)^*}=\zeta((z^2y_2)^*)
\end{eqnarray*}
then, using the poly-morphism $\zeta$, we obtain
\begin{eqnarray*}
\zeta((-z^4y_4)^*)
&=&\zeta((-z^2y_2)^*)\zeta((z^2y_2)^*)\\
&=&\zeta((-z^2x_0x_1)^*)\zeta((z^2x_0x_1)^*))\\
&=&\zeta((-4z^4x_0^2x_1^2)^*).
\end{eqnarray*}
It follows then, by identification the coeffients of $z^{4k}$, that
\begin{eqnarray*}
\dfrac{\zeta(\overbrace{{3,1,\ldots,3,1}}^{k{\tt times}})}{\pi^{4k}}
=\dfrac{4^k\zeta(\overbrace{{4,\ldots,4}}^{k{\tt times}})}{\pi^{4k}}
=\dfrac{2}{(4k+2)!}\in\Q.
\end{eqnarray*}
\end{itemize}
\end{example}

\section{Conclusion}
Noncommutative symbolic calculus allows to get identities easy to check and to implement. 
With some amount of complex and functional analysis, it is possible to bridge the gap between symbolic,
functional and number theoretic worlds. This was the case already for polylogarithms and polyzetas. 
This is the project of this paper and will be pursued in forthcoming works.
\section*{Acknowledgment}
The authors would like to thank many colleagues for detailed reading of the manuscript
and many valuable comments by which this paper could be improved.


\end{document}